\documentclass{amsart}
\usepackage{preamble}

\usepackage{parskip}
\numberwithin{equation}{section}

\newcommand{\altref}[2]{%
    \textup{(}\hyperref[#1]{\textup{#2}}\textup{)}%
}

\newcommand{\sint}[0]{\textstyle\int}
\newcommand{\ssum}[0]{\textstyle\sum}

\bibliography{refs}

\title{Invariant measures for the open KPZ equation: the Gaussian case}
\author{James Bona-Landry}
\address{Department of Mathematics, University of Toronto}
\email{james.bonalandry@mail.utoronto.ca}

\begin{document}
    \begin{abstract}
        In \cite{gu2024integrationpartsinvariantmeasure}, Quastel and Gu use Stein's equation and integration by parts to give a direct proof that drifted Brownian motions are stationary (modulo height shifts) for the full-line KPZ equation. In this article, we consider the open KPZ equation with boundary conditions $\partial_x h(t,0) = \partial_x h(t,1) = \alpha$ for a general real parameter $\alpha \in \R$, and emulate the approach of Quastel and Gu to provide a similar proof that Brownian motion with constant drift $\alpha$ is invariant (modulo height shifts) in this case.
    \end{abstract}
    \maketitle

    \section{Introduction}\label{sec:introduction}

    \subsection{Background}\label{sec:introduction, subsec:background}
    The open Kardar-Parisi-Zhang (KPZ) equation is a stochastic partial differential equation modeling the random evolution of an interface $h \in C(\R_+ \times [0,1])$ subject to Neumann-type boundary conditions determined by two real parameters $\alpha,\beta \in \R$. Formally, it is written as
    \begin{equation}\label{eqn:KPZ}
        \partial_t h = \tfrac{1}{2}(\partial_x h)^2 + \tfrac{1}{2}\partial_x^2 h + \xi, \quad \partial_x h(t,0) = \alpha, \ \partial_x h(t,1) = -\beta, \quad h(0,x) = h_0(x) \tag{$\text{KPZ}_{\alpha,\beta}$}
    \end{equation}
    where $\xi$ denotes space-time white noise, the distribution-valued Gaussian field on $\R_+ \times [0,1]$ that is delta-correlated in space and time. Due to the singular nature of $\xi$, the rigorous interpretation of \eqref{eqn:KPZ} is a non-trivial task. Based on the linear theory, we expect the ``solution'' $h$ to have spatial regularity at most H\"older-$\frac{1}{2}-$; the spatial derivative $\partial_x h$ then exists only as a distribution, rendering the nonlinearity and pointwise boundary conditions appearing in \eqref{eqn:KPZ} ill-defined. A general framework for circumventing these issues and obtaining a rigorous solution theory is provided by Hairer's work on regularity structures \cite{Hairer2013, Gerencs_r_2018}. A simpler approach that works in the case of the open KPZ equation, however, (and which leads to the same result as Hairer's framework) is to define the solution as $h(t,x) := \log Z(t,x)$ for $Z \in C(\R_+ \times [0,1])$ satisfying the open stochastic heat equation
    \begin{equation}\label{eqn:SHE}
        \partial_t Z = \tfrac{1}{2}\partial_x^2 Z + Z\xi, \quad \partial_x Z(t,0) = \hat{\alpha} Z(t,0), \ \partial_x Z(t,1) = -\hat{\beta} Z(t,1), \quad Z(0,x) = e^{h_0(x)} \tag{$\text{SHE}_{\hat{\alpha},\hat{\beta}}$}
    \end{equation}
    in the It\^{o}-mild sense with $\hat{\alpha} = \alpha - \frac{1}{2}$ and $\hat{\beta} = \beta - \frac{1}{2}$. This solution is referred to as the \emph{Hopf-Cole} solution to \eqref{eqn:KPZ}. Note that there is a shift by $\frac{1}{2}$ in the boundary parameters here; although inconsistent with what would be suggested by a formal calculation, these shifts are motivated by \cite{Goncalves2020} and have become standard convention in the literature. Very roughly, they arise from viewing $h$ as the scaling limit of a sequence of approximating WASEPs, in which the formal boundary conditions have a natural physical interpretation.

    In recent years, one of the central focuses of research activity surrounding \eqref{eqn:KPZ} has been characterizing and understanding its (almost) stationary distributions. The world ``almost'' here reflects the fact these distributions are not truly stationary (due to the fact that $h$ has a deterministic upward drift), but rather invariant modulo height shifts. Descriptions of varying degrees of complexity were obtained by \cite{corwin2023stationarymeasureopenkpz, Bryc_2022, Barraquand_2022}, and it was ultimately shown in \cite{Bryc_2022} that if $\alpha + \beta \geq 0$ and $\min(\alpha,\beta) > -1$, then there is a unique stationary distribution $\mu_{\alpha,\beta}$ for the height-shifted process $(h(t,x)-h(t,0))_{x \in [0,1]}$ that decomposes into a sum of independent processes as $\mu_{\alpha,\beta} =_{\mrm{law}} W + X$, where $W \in C([0,1])$ is variance $\frac{1}{2}$ Brownian motion, and $X \in C([0,1])$ is absolutely continuous in law with respect to variance $\frac{1}{2}$ Brownian motion, with Radon-Nikodym derivative proportional to
    \begin{align*}
        e^{-\beta X(1)}\left(\int_{0}^{1} e^{-2X(x)} dx\right)^{-(\alpha+\beta)}.
    \end{align*}
    Notice that in the special case $\alpha + \beta = 0$, the nonlinearity collapses and $X$ reduces to a Gaussian process, specifically variance $\frac{1}{2}$ Brownian motion with constant drift $\alpha = -\beta$. By independence, the invariant state $\mu_{\alpha,-\alpha} = W + X$ then becomes standard Brownian motion with constant drift $\alpha$. The goal of this paper is to provide a ``simple'' proof (in the sense of being, modulo a key symmetry, a series of straightforward calculations) of invariance in this case by emulating the approach taken by Quastel and Gu in \cite{gu2024integrationpartsinvariantmeasure}, which leverages the Gaussian nature of the distribution.

    \subsection{Organization}\label{sec:introduction, subsec:organization}
    As an application of their approach to an analogous setting, the organization (and notation) of this paper has been deliberately chosen so as to mimic that of the blueprint paper \cite{gu2024integrationpartsinvariantmeasure} of Quastel and Gu. We begin by outlining the approach and general setup in Section \ref{sec:introduction, subsec:method} before presenting the main results in Section \ref{sec:main results}. Sections \ref{sec:proof of auxiliary result} and \ref{sec:key symmetry} are then dedicated to the proofs of these main results, with the latter focusing on a key symmetry, Proposition \ref{prop:key symmetry}, on which these proofs rely.

    \subsection{Integration by parts method}\label{sec:introduction, subsec:method}
    Because the invariance of $\mu_{\alpha,\beta}$ is only up to height shifts, it will be convenient to work with $u := \partial_x h$ instead of $h$, which is defined as the solution to the open stochastic Burgers' equation
    \begin{equation}\label{eqn:SBE}
        \partial_t u = \tfrac{1}{2}\partial_x u^2 + \tfrac{1}{2}\partial_x^2 u + \partial_x \xi, \quad u(t,0) = \alpha, \ u(t,1) = -\beta, \quad u(0,x) = u_0(x). \tag{$\text{SBE}_{\alpha,\beta}$}
    \end{equation}
    This is a stochastic process in the space of distributions, and the ``almost'' invariance of standard Brownian motion with drift $\alpha$ under the flow of \altref{eqn:KPZ}{KPZ$_{\alpha,-\alpha}$} is equivalent to the true invariance of mean $\alpha$ (spatial) white noise under the flow of \altref{eqn:SBE}{SBE$_{\alpha,-\alpha}$}. Our goal will thus be to show that if \altref{eqn:SBE}{SBE$_{\alpha,-\alpha}$} is initialized with mean $\alpha$ spatial white noise, then the solution remains as this at all later times $t > 0$. Equivalently, $\tilde{u}_t := u_t - \alpha$ is spatial white noise for every $t > 0$. By Stein's equation, this will follow if we can show that
    \begin{equation}\label{eqn:steins}
        \bb{E}[F(\inn{\tilde{u}_t,f}) \inn{\tilde{u}_t,f}] = \norm{f}_{L^2([0,1])}^2 \bb{E}F'(\inn{\tilde{u}_t,f}), \quad \text{for every} \ f \in C_c^{\infty}([0,1]), \ F \in C_c^{\infty}(\R),
    \end{equation}
    where $\bb{E} = \mbf{E}\bb{E}^{\mrm{init}}$ denotes expectation over the probability space $\Omega = \Omega^{\mrm{noise}} \times \Omega^{\mrm{init}}$ associated to $u_t$. To do so, we approximate $u$ by a smooth field for which we have access to a nice polymer representation, and compute the left-hand side of \eqref{eqn:steins} making use of the following integration by parts identity from the Malliavin calculus:
    
    \begin{proposition}[\cite{nualart2006malliavin}, Chapter 1]\label{prop:malliavin integration by parts}
        Let $W$ be either space-time white noise on $\mathbb{R}_+ \times [0,1]$ or spatial white noise on $[0,1]$. Then for any deterministic $f \in L^2$ on the corresponding domain and any random variable $F$ in the Malliavin-Sobolev space $\bb{D}^{1,2}$, we have that
        \begin{align*}
            \bb{E}\left[F \inn{W,f}\right] = \bb{E}\left[\inn{\ms{D}F, f}_{L^2}\right],
        \end{align*}
        where $\ms{D}$ denotes the Malliavin derivative operator associated to $W$.
    \end{proposition}

    Our particular choice of approximation for $u = \partial_x(\log Z)$ is motivated by the formal Feynman-Kac formula
    \begin{equation}\label{eqn:SHE feynman-kac representation}
        Z(t,x) = \bb{E}^{\mrm{RBM}}_x\left[e^{\int_{0}^{t} \xi(t-s,R_s) ds - \frac{1}{2}\var\left(\int_{0}^{t} \xi(t-s,R_s) ds\right) - \frac{1}{2}\left(\hat{\alpha}\ms{L}^0_t + \hat{\beta}\ms{L}^1_t\right) + h_0(R_t)}\right],
    \end{equation}
    which would make sense and characterize the solution to \eqref{eqn:SHE} if $\xi$ and $h_0$ were sufficiently ``nice'' (see, e.g. \cite{Lejay_2016, scorolli2021feynmankacformulaheatequation}). Here, $\bb{E}^{\mrm{RBM}}_x$ denotes expectation with respect to a reflecting Brownian motion $(R_t)_{t \geq 0}$ on $[0,1]$ started from $x$, and $(\ms{L}^y_t)_{y \in [0,1], t \geq 0}$ denotes the jointly continuous version of the local time process associated to $R$, characterized by the occupation density formula
    \begin{align*}
        \int_{0}^{t} f(R_s) ds = \int_{0}^{1} f(y) \ms{L}^y_t dy \quad \text{for all bounded measurable $f : [0,1] \to \R$.}
    \end{align*}
    Specifically, for $\beta = -\alpha$ and $u_0$ given by mean $\alpha$ spatial white noise (and $h_0$ determined as $h_0(x) = \int_{0}^{x} u_0(y) dy$), we introduce regularized approximations of the terms appearing in \eqref{eqn:SHE feynman-kac representation} so as to allow us to take a logarithmic derivative in space and obtain a Feynman-Kac-type formula approximating $u$. More precisely, we let $\ep,\kappa > 0$ denote smoothing parameters -- one for the forcing noise and local time terms, and the other for the initial data -- that will characterize these approximations and, in the end, will be taken to zero separately to recover $u$. To make explicit the dependence of certain objects on these parameters, the pair $(\ep,\kappa)$ will be denoted by $\theta$ and included as superscript or subscript where appropriate. With this in mind, we define
    \begin{align*}
        \xi_{\ep}(t,x) := (e^{\ep\Delta_{\mrm{neu}}} \xi(t,\cdot))(x), \quad h_0^{\kappa}(x) := (e^{\kappa\Delta_{\mrm{neu}}} h_0)(y) dy,
    \end{align*}
    where $(e^{t\Delta_{\mrm{neu}}})_{t \geq 0}$ denotes the heat semigroup generated by the Laplacian on $[0,1]$ with homogeneous Neumann boundary conditions.
    
    \begin{remark}
        To briefly motivate this particular choice of smoothing, it shall be computationally convenient to ``unfold'' the problem and view reflecting Brownian motion on $[0,1]$ in terms of an underlying ``driving'' Brownian motion on $\R$ as $R_t = \varrho(B_t+x)$, where $\varrho : \R \to [0,1]$ denotes the reflection function
        \begin{equation}\label{eqn:reflection function}
            \varrho(x) = \begin{cases}
                x - \floor{x} & \text{if $\floor{x}$ is even}, \\
                \floor{x}+1 - x & \text{if $\floor{x}$ is odd}.
            \end{cases}
        \end{equation}
        Crucially, the Neumann heat kernel $p^{\mrm{neu}}(t,x;y) := (e^{t\Delta_{\mrm{neu}}} \delta_y)(x)$ behaves very nicely with respect to this reflection function. Indeed, when identified with its natural extension to $\R_+ \times \R^2$ via the expansion
        \begin{equation}\label{eqn:neumann heat kernel}
            p^{\mrm{neu}}(t,x;y) = \sum_{n \in \Z} \left(p(t,x-y-2n) + p(t,x+y-2n)\right), \quad p(t,x) := (2\pi t)^{-1/2} e^{-x^2/2t},
        \end{equation}
        it is even and 2-periodic -- and hence $\varrho$-invariant -- in both spatial variables. It is easy to see that any function of the form $f = e^{t\Delta_{\mrm{neu}}} f_0$ is then likewise $\varrho$-invariant, allowing us to freely move between the reflecting process $R$ and the underlying non-reflecting process $B$.
    \end{remark}
    
    For the local time terms, we do something slightly simpler and observe that since $p^{\mrm{neu}}(t,x;y)$ approximates $\delta_y$ in the weak sense on $C([0,1])$ as $t \to 0$, we have that $\int_{0}^{T} p^{\mrm{neu}}(t, R_s; y) ds = \int_{0}^{1} p^{\mrm{neu}}(t, x; y) \ms{L}^x_T dx \to \ms{L}^y_T$ as $t \to 0$ for every $y \in [0,1]$, and hence we can smoothly approximate $\ms{L}^0_t$ and $\ms{L}^1_t$ via $\int_{0}^{t} p^{\mrm{neu}}(\ep, R_s; 0) ds$ and $\int_{0}^{t} p^{\mrm{neu}}(\ep, R_s; 1) ds$, respectively. With this, we define our approximating field $Z^{\theta}$ as
    \begin{equation}\label{eqn:approximation to Z}
        Z^{\theta}(t,x) := \bb{E}^{\mrm{RBM}}_x\left[e^{\int_{0}^{t} \xi_{\ep}(t-s,R_s) - \frac{1}{2}C_{\ep}(R_s) ds + \int_{0}^{t} V_{\ep}(R_s) ds + h_0^{\kappa}(R_t)}\right],
    \end{equation}
    where $C_{\ep}(x) := p^{\mrm{neu}}(2\ep,x;x)$, $V_{\ep}(x) := -\tfrac{1}{2}\left(\hat{\alpha} p^{\mrm{neu}}(\ep,x;0) + \hat{\beta} p^{\mrm{neu}}(\ep,x;1)\right)$, and $\alpha + \beta = 0$. Here, we have written the normalization term for the stochastic integral pathwise as $C_{\ep}(R_s)$ simply for computational convenience. Since everything appearing in \eqref{eqn:approximation to Z} is by construction $\varrho$-invariant, we can now freely replace every instance of $R_s$ with $B_s+x$ and take a logarithmic derivative in $x$ to obtain an approximation to $u$:
    \begin{equation}\label{eqn:approximation to u, expanded form}
        u^{\theta}(t,x) := \frac{\bb{E}^{\mrm{BM}}_x\left[e^{-H^{\theta}(t,B)}\left(\int_{0}^{t} \partial_x \xi_{\ep}(t-s,B_s) - \frac{1}{2}C_{\ep}'(B_s) ds + \int_{0}^{t} V_{\ep}'(B_s) ds + u_0^{\kappa}(B_t)\right)\right]}{Z^{\theta}(t,x)},
    \end{equation}
    where $u_0^{\kappa} := \partial_x h_0^{\kappa}$ and we have used the abbreviation
    \begin{align*}
        -H^{\theta}(t,X) &:= \int_{0}^{t} \xi_{\ep}(t-s,X_s) - \tfrac{1}{2}C_{\ep}(X_s) ds + \int_{0}^{t} V_{\ep}(X_s) ds + h_0^{\kappa}(X_t).
    \end{align*}
    By viewing $\frac{e^{-H^{\theta}(t,B)}}{Z^{\theta}(t,x)}$ as the density of a ``polymer measure'' $\bb{P}^{\mrm{poly},\theta}_{t,x}$ with respect to the Wiener measure on continuous paths $B \in C_x(\R_+)$ started from $x$, we write this more compactly as
    \begin{equation}\label{eqn:approximation to u, polymer form}
        u^{\theta}(t,x) = \bb{E}^{\mrm{poly},\theta}_{t,x}\left[\int_{0}^{t} \partial_x \xi_{\ep}(t-s,X_s) - \tfrac{1}{2}C_{\ep}'(X_s) ds + \int_{0}^{t} V_{\ep}'(X_s) ds + u_0^{\kappa}(X_t)\right].
    \end{equation}
    The limiting Gibbs-type measure obtained by taking the smoothing parameters $\ep,\kappa \to 0$ shall be denoted by $\bb{P}^{\mrm{poly}}_{t,x}$, with corresponding expectation $\bb{E}^{\mrm{poly}}_{t,x}$.

    \section{Main Results}\label{sec:main results}

    \begin{theorem}\label{thm:invariance}
        For every $\alpha \in \R$, the flow of \altref{eqn:SBE}{SBE$_{\alpha,-\alpha}$} preserves mean $\alpha$ white noise: if $u_0$ is distributed as mean $\alpha$ white noise, then so too is $u_t$ for every $t > 0$.
    \end{theorem}

    Our principal instrument for establishing Theorem \ref{thm:invariance} will be the following auxiliary result, Proposition \ref{prop:auxiliary result}. Before stating it, we introduce some notation that will be convenient going forward. Firstly, $\zeta$ shall denote the zig-zag function $\R \to (-1,1)$ which is zero on the integers and which decreases linearly from $1$ to $-1$ on each open interval $(n, n+1)$. Next, $\sigma$ shall denote the alternating sign function $\R \to \{-1,0,1\}$ which takes the value $+1$ on ``even'' intervals $(2n, 2n+1)$, $-1$ on ``odd'' intervals $(2n-1, 2n)$, and zero on the integers. For a given parameter $p > 0$, the notation $\zeta_p$ and $\sigma_p$ will be used to denote smooth functions that converge pointwise to $\zeta$ and $\sigma$, respectively, as $p \to 0$. Finally, $\bb{E}^{\mrm{poly},\theta}_{t,x_1,x_2}$ shall denote expectation with respect to the product measure $\bb{P}^{\mrm{poly},\theta}_{t,x_1} \times \bb{P}^{\mrm{poly},\theta}_{t,x_2}$, i.e. with respect to two independent polymer paths $X^1$ and $X^2$ started from $x_1$ and $x_2$, respectively.

    \begin{proposition}\label{prop:auxiliary result}
        Given $\alpha \in \R$ and $t > 0$, let $u^{\theta}_t = u^{\theta}(t,\cdot)$ denote the approximating field \eqref{eqn:approximation to u, expanded form}/\eqref{eqn:approximation to u, polymer form} at time $t > 0$ and let $\tilde{u}^{\theta}_t = u^{\theta}_t - \alpha$. Moreover, given $f \in C_c^{\infty}([0,1])$, let $Y = \inn{\tilde{u}^{\theta}_t,f}$. Then, for any $F \in C_c^{\infty}(\R)$, it holds that
        \begin{equation}\label{eqn:steins for approximating field}
            \begin{split}
                \bb{E} F(Y) Y &= \bb{E} F'(Y) \int_{[0,1]^2} dx_1 dx_2 \ f(x_1) f'(x_2)\left(\tfrac{1}{2}\zeta_{2\ep}(\tfrac{1}{2}(x_1-x_2)) + \tfrac{1}{2}\zeta_{2\ep}(\tfrac{1}{2}(x_1+x_2))\right) \\
                &\quad\quad + \Gamma^{\mrm{noise}} + \Gamma^{\mrm{bdy}} + \Gamma^{\mrm{init}} + \mc{E},
            \end{split}
        \end{equation}
        where $\mc{E}$ is an error term that vanishes as $\ep \to 0$, and $\Gamma^{\mrm{noise}}$, $\Gamma^{\mrm{bdy}}$, and $\Gamma^{\mrm{init}}$ are defined as in \eqref{eqn:contribution from noise}, \eqref{eqn:contribution from boundary data}, and \eqref{eqn:contribution from initial data}, respectively.
    \end{proposition}

    The proof of Proposition \ref{prop:auxiliary result} is a series of relatively straightforward calculations and is deferred to the next section. We now complete the proof of Theorem \ref{thm:invariance}, given Proposition \ref{prop:auxiliary result}. By taking the smoothing parameters $\ep,\kappa \to 0$ in \eqref{eqn:contribution from noise}-\eqref{eqn:contribution from initial data} we get a series of cancellations and $\Gamma^{\mrm{noise}} + \Gamma^{\mrm{bdy}} + \Gamma^{\mrm{init}}$ becomes
    \begin{equation}\label{eqn:total correction term, limit}
        \begin{split}
            -&\bb{E} F'(\inn{\tilde{u}_t,f}) \int_{[0,1]^2} dx_1 dx_2 \ f(x_1) f'(x_2) \\
            &\hspace{96pt}\cdot \bb{E}^{\mrm{poly}}_{t,x_1,x_2}\left[\tfrac{1}{2}\zeta(\tfrac{1}{2}(X^1_t-X^2_t)) + \tfrac{1}{2}\zeta(\tfrac{1}{2}(X^1_t+X^2_t)) + \sigma(X^1_t) \mbf{1}_{[0,\varrho(X^2_t)]}(\varrho(X^1_t))\right] \\
            &\quad\quad -\bb{E} F(\inn{\tilde{u}_t,f}) \int_{0}^{1} dx \ f(x) \bb{E}^{\mrm{poly}}_{t,x,x}\left[\tfrac{1}{2}\zeta(\tfrac{1}{2}(X_t-\tilde{X}_t)) + \tfrac{1}{2}\zeta(\tfrac{1}{2}(X_t+\tilde{X}_t)) + \sigma(X_t) \mbf{1}_{[0,\varrho(\tilde{X}_t)]}(\varrho(X_t))\right] \\
            &\quad\quad + \bb{E} F(\inn{\tilde{u}_t,f}) \int_{0}^{1} dx \ f(x) \bb{E}^{\mrm{poly}}_{t,x}\left[\tfrac{1}{2}\zeta(X_t) + \tfrac{1}{2}\sigma(X_t)\right],
        \end{split}
    \end{equation}
    where $\tilde{u}_t$ is defined as in Section \ref{sec:introduction, subsec:method}, i.e. the solution to \altref{eqn:SBE}{SBE$_{\alpha,-\alpha}$} at time $t$ translated vertically by a factor of $-\alpha$. Next, it is straightforward to check that (almost everywhere) we have the identities
    \begin{align*}
        \left(\tfrac{1}{2}\zeta(\tfrac{1}{2}(x_1-x_2)) + \tfrac{1}{2}\zeta(\tfrac{1}{2}(x_1+x_2))\right) + \sigma(x_1) \mbf{1}_{[0,\varrho(x_2)]}(\varrho(x_1)) = \zeta(\tfrac{1}{2}x_1), \quad \tfrac{1}{2}\zeta(x) + \tfrac{1}{2}\sigma(x) = \zeta(\tfrac{1}{2}x).
    \end{align*}
    By making these substitutions, the second and third terms in \eqref{eqn:total correction term, limit} cancel exactly, and the first is killed by integrating by parts in $x_2$, so that in the limit $\ep,\kappa \to 0$ we have $\Gamma^{\mrm{noise}} + \Gamma^{\mrm{bdy}} + \Gamma^{\mrm{init}} = 0$. On the other hand, for the first term on the right-hand side of \eqref{eqn:steins for approximating field}, we note that $\frac{1}{2}\zeta(\frac{1}{2}\cdot)$ agrees with $\frac{1}{2}(\sgn(x)-x)$ identically on $(-2,2)$, so that in the limit $\ep \to 0$ we can make the replacement
    \begin{align*}
        \tfrac{1}{2}\zeta(\tfrac{1}{2}(x_1-x_2)) + \tfrac{1}{2}\zeta(\tfrac{1}{2}(x_1+x_2)) \to \tfrac{1}{2}\sgn(x_1-x_2) + \tfrac{1}{2}\sgn(x_1+x_2) - x_1.
    \end{align*}
    Integrating by parts in $x_2$, we then obtain
    \begin{align*}
        \bb{E} F(\inn{\tilde{u}_t,f}) \inn{\tilde{u}_t,f} &= \norm{f}_{L^2([0,1])}^2 \bb{E} F'(\inn{\tilde{u}_t,f}),
    \end{align*}
    By Stein's equation \eqref{eqn:steins}, Theorem \ref{thm:invariance} follows.

    \section{Proof of Proposition \ref{prop:auxiliary result}}\label{sec:proof of auxiliary result}

    By distributing the polymer expectation in \eqref{eqn:approximation to u, polymer form} onto each of its three arguments, we divide the proof into three separate (and relatively short) calculations: a contribution from the noise,
    \begin{equation}\label{eqn:calculation, noise part}
        \bb{E} F(Y) \int_{0}^{1} dx \ f(x) \bb{E}^{\mrm{poly},\theta}_{t,x}\left[\int_{0}^{t} \partial_x \xi_{\ep}(t-s,X_s) - \tfrac{1}{2}C_{\ep}'(X_s) ds\right],
    \end{equation}
    a contribution from the boundary data,
    \begin{equation}\label{eqn:calculation, boundary part}
        \bb{E} F(Y) \int_{0}^{1} dx \ f(x) \bb{E}^{\mrm{poly},\theta}_{t,x}\left[\int_{0}^{t} V_{\ep}'(X_s) ds\right],
    \end{equation}
    and, finally, a contribution from the initial data,
    \begin{equation}\label{eqn:calculation, initial data part}
        \bb{E} F(Y) \int_{0}^{1} dx \ f(x) \bb{E}^{\mrm{poly},\theta}_{t,x} u_0^{\kappa}(X_t).
    \end{equation}
    With this, we have the identity
    \begin{equation}\label{eqn:lhs of steins}
        \bb{E} F(Y) Y = \eqref{eqn:calculation, noise part} + \eqref{eqn:calculation, boundary part} + \eqref{eqn:calculation, initial data part} - \alpha \bb{E}F(Y) \int_{0}^{1} f(x) dx.
    \end{equation}
    The calculations of \eqref{eqn:calculation, noise part} and \eqref{eqn:calculation, boundary part} will hinge on the following non-trivial key symmetry, which we prove in Section \ref{sec:key symmetry}:

    \begin{proposition}\label{prop:key symmetry}
        Let $R_{\ep} : \R \to \R$ denote the 2-periodic function $x \mapsto \sum_{n \in \Z} p(\ep,x-2n)$ for $p(t,x) = (2\pi t)^{-1/2}e^{-x^2/2t}$ given by the standard heat kernel. There is a finite constant $C = C(t,\kappa) > 0$ such that
        \begin{align*}
            \mbf{E}\left|\bb{E}^{\mrm{poly},\theta}_{t,x_1,x_2} \int_{0}^{t} R_{\ep}(X^1_s - X^2_s) d(X^1_s - X^2_s)\right| \leq C \ep^{1/4} |\log \ep|^2.
        \end{align*}
        In particular, as $\ep \to 0$ (and $t$ and $\kappa$ are kept fixed) this quantity vanishes in $L^1(\Omega^{\mrm{noise}})$ uniformly in $x_1$ and $x_2$.
    \end{proposition}
    \begin{remark}
        We shall see in the following calculations that terms of the form $\bb{E}^{\mrm{poly,\theta}}_{t,x_1,x_2} \int_{0}^{t} R_{\ep}(X^1_s+X^2_s) d(X^1_s + X^2_s)$ and $\bb{E}^{\mrm{poly,\theta}}_{t,x} \int_{0}^{t} R_{\ep}(X_s) dX_s$ arise as well, that we also claim vanish as $\ep \to 0$. The is easily seen for the former by making the change of variables $X^2 \mapsto \tilde{X}^2 := -X^2$, which reduces it to the case of Proposition \ref{prop:key symmetry} above. For the latter, there is no direct way to make a similar reduction, however it is not difficult to see that every step in the proof of Proposition \ref{prop:key symmetry} applies in this case as well (in fact, most are even simplified a bit).
    \end{remark}
    
    \subsection{Contribution from the noise}\label{sec:proof of auxiliary result, subsec:noise contribution}
    We will begin the calculation of \eqref{eqn:calculation, noise part} with $\mbf{E}$ in place of $\bb{E}$ (for a fixed realization of the initial data); at the end, we will take expectations $\bb{E}^{\mrm{init}}$ with respect to the initial data to recover $\bb{E}$. With this in mind, \eqref{eqn:calculation, noise part} breaks into two components,
    \begin{equation}\label{eqn:calculation, constituent terms in noise contribution}
        \mbf{E} F(Y) \int_{0}^{1} dx \ f(x) \bb{E}^{\mrm{poly},\theta}_{t,x} \int_{0}^{t} \partial_x \xi_{\ep}(t-s,X_s) ds - \tfrac{1}{2}\mbf{E} F(Y) \int_{0}^{1} dx \ f(x) \bb{E}^{\mrm{poly},\theta}_{t,x} \int_{0}^{t} C_{\ep}'(X_s) ds.
    \end{equation}
    By expanding the polymer expectation and applying the integration by parts identity Proposition \ref{prop:malliavin integration by parts}, the first term here becomes
    \begin{align*}
        \mbf{E} F(Y) \int_{0}^{1} dx \ f(x) &\bb{E}^{\mrm{BM}}_x \tfrac{e^{-H^{\theta}(t,B)}}{Z^{\theta}(t,x)} \int_{0}^{t} ds \int_{0}^{1} dy \ \partial_x p^{\mrm{neu}}(\ep,y; B_{t-s}) \xi(s,y) \\
        &= \mbf{E} \int_{0}^{1} dx \ f(x) \bb{E}^{\mrm{BM}}_x \int_{0}^{t} ds \int_{0}^{1} dy \ \partial_x p^{\mrm{neu}}(\ep,y; B_{t-s}) \ms{D}_{s,y}\left(F(Y)\tfrac{e^{-H^{\theta}(t,B)}}{Z^{\theta}(t,x)}\right).
    \end{align*}
    To compute the Malliavian derivative, we apply product rule together with the easily checked identities
    \begin{align*}
        \ms{D}_{s,y} Y = -\int_{0}^{1} dx_2 \ f'(x_2) \tfrac{\ms{D}_{s,y} Z^{\theta}(t,x_2)}{Z^{\theta}(t,x_2)}, \quad \ms{D}_{s,y}(-H^{\theta}(t,B)) = p^{\mrm{neu}}(\ep, y; B_{t-s})
    \end{align*}
    yielding
    \begin{align*}
        \ms{D}_{s,y}\left(F(Y)\tfrac{e^{-H^{\theta}(t,B)}}{Z^{\theta}(t,x)}\right) &= -F'(Y)\tfrac{e^{-H^{\theta}(t,B)}}{Z^{\theta}(t,x)} \int_{0}^{1} dx_2 \ f'(x_2) \bb{E}^{\mrm{poly},\theta}_{t,x_2} p^{\mrm{neu}}(\ep,y; B^2_{t-s}) \\
        &\quad\quad + F(Y)\tfrac{e^{-H^{\theta}(t,B)}}{Z^{\theta}(t,x)}\left(p^{\mrm{neu}}(\ep,y; B_{t-s}) - \bb{E}^{\mrm{poly},\theta}_{t,x} p^{\mrm{neu}}(\ep, y; \tilde{B}_{t-s})\right).
    \end{align*}
    By now substituting back in, performing the integrations in $y$ via Chapman-Kolmogorov, and changing variables $t-s \mapsto s$, we obtain:
    \begin{align*}
        \mbf{E} &F(Y) \int_{0}^{1} dx \ f(x) \bb{E}^{\mrm{poly,\theta}}_{t,x} \int_{0}^{t} ds \ \partial_x \xi_{\ep}(t-s,X_s) \\
        &= -\mbf{E} F'(Y) \int_{[0,1]^2} dx_1 dx_2 \ f(x_1) f'(x_2) \bb{E}^{\mrm{poly},\theta}_{t,x_1,x_2} \int_{0}^{t} ds \ \partial_x p^{\mrm{neu}}(2\ep, X^1_s; X^2_s) \\
        &\quad\quad + \mbf{E} F(Y) \int_{0}^{1} dx \ f(x) \bb{E}^{\mrm{poly},\theta}_{t,x}\int_{0}^{t} ds \ \left(\partial_x p^{\mrm{neu}}(2\ep, X_s; X_s) - \bb{E}^{\mrm{poly},\theta}_{t,x} \partial_x p^{\mrm{neu}}(2\ep, X_s; \tilde{X}_s)\right).
    \end{align*}
    On the other hand, the second term in \eqref{eqn:calculation, constituent terms in noise contribution} is easily computed as
    \begin{align*}
        -\mbf{E} F(Y) \int_{0}^{1} dx \ f(x) \bb{E}^{\mrm{poly},\theta}_{t,x}\int_{0}^{t} ds \ \partial_x p^{\mrm{neu}}(2\ep, X_s; X_s),
    \end{align*}
    so the total contribution from the noise \eqref{eqn:calculation, noise part} becomes:
    \begin{align*}
        -\mbf{E} F'(Y) \int_{0}^{1} dx_1 \int_{0}^{1} dx_2 \ &f(x_1) f'(x_2) \bb{E}^{\mrm{poly}}_{t,x_1,x_2} \int_{0}^{t} ds \ \partial_x p^{\mrm{neu}}(2\ep, X^1_s; X^2_s) \\
        &- \mbf{E} F(Y) \int_{0}^{1} dx \ f(x) \bb{E}^{\mrm{poly}}_{t,x,x} \int_{0}^{t} ds \ \partial_x p^{\mrm{neu}}(2\ep, X_s; \tilde{X}_s).
    \end{align*}
    We now compute the integrations in time via It\^{o}'s formula. To do so, we note that if $R_t(x) := \sum_{n \in \Z} p(t,x-2n)$ and $r_t(x) := \int_{0}^{x} R_t(y) - \frac{1}{2} dy$ (the $\frac{1}{2}$ here is just to ensure that $r$ inherits the 2-periodicity of $R$, which will be computationally convenient for us), then from the expansion \eqref{eqn:neumann heat kernel} it follows that
    \begin{align*}
        \partial_x p^{\mrm{neu}}(t,x;y) = R_t'(x-y) + R_t'(x+y) = r_t''(x-y) + r_t''(x+y).
    \end{align*}
    Subsequently, by It\^{o}'s formula we have
    \begin{align*}
        \int_{0}^{t} \partial_x p^{\mrm{neu}}(2\ep, X^1_s; X^2_s) ds &= r_{2\ep}(X^1_t - X^2_t) - r_{2\ep}(x_1-x_2) - \int_{0}^{t} \left(R_{2\ep}(X^1_s - X^2_s) - \tfrac{1}{2}\right) d(X^1_s - X^2_s) \\
        &\quad\quad + r_{2\ep}(X^1_t + X^2_t) - r_{2\ep}(x_1+x_2) - \int_{0}^{t} \left(R_{2\ep}(X^1_s + X^2_s) - \tfrac{1}{2}\right) d(X^1_s + X^2_s) \\
        &= \left(r_{2\ep}(X^1_t - X^2_t) + r_{2\ep}(X^1_t + X^2_t)\right) - \left(r_{2\ep}(x_1-x_2) + r_{2\ep}(x_1+x_2)\right) + X^1_t + \cdots
    \end{align*}
    In the final line here, we have chosen to omit the stochastic integral terms because they are immaterial in the limit $\ep \to 0$ by Proposition \ref{prop:key symmetry}. Similarly, we have 
    \begin{align*}
        \int_{0}^{t} \partial_x p^{\mrm{neu}}(2\ep, X_s; \tilde{X}_s) ds &= \left(r_{2\ep}(X_t - \tilde{X}_t) + r_{2\ep}(X_t + \tilde{X}_t)\right) - r_{2\ep}(2x) + X_t + \cdots
    \end{align*}
    Finally, it is easily checked that that $r_t \to \frac{1}{2}\zeta(\frac{1}{2}\cdot)$ as $t \to 0$ (where $\zeta$ is defined as in the first paragraph in Section \ref{sec:main results}); for this reason, we write $r_{2\ep} = \frac{1}{2}\zeta_{2\ep}(\frac{1}{2}\cdot)$ and take expectations $\bb{E}^{\mrm{init}}$ with respect to the initial data to conclude finally that
    \begin{align*}
        \eqref{eqn:calculation, noise part} &= \bb{E} F'(Y) \int_{[0,1]^2} dx_1 dx_2 \ f(x_1) f'(x_2)\left(\tfrac{1}{2}\zeta_{2\ep}(\tfrac{1}{2}(x_1-x_2)) + \tfrac{1}{2}\zeta_{2\ep}(\tfrac{1}{2}(x_1+x_2))\right) + \Gamma^{\mrm{noise}} + \mc{E}^{\mrm{noise}},
    \end{align*}
    where $\mc{E}^{\mrm{noise}} \to 0$ as $\ep \to 0$ and
    \begin{equation}\label{eqn:contribution from noise}
        \begin{split}
            \Gamma^{\mrm{noise}} &= -\bb{E} F'(Y) \int_{[0,1]^2} dx_1 dx_2 \ f(x_1) f'(x_2) \bb{E}^{\mrm{poly},\theta}_{t,x_1,x_2}\left[\tfrac{1}{2}\zeta_{2\ep}(\tfrac{1}{2}(X^1_t-X^2_t)) + \tfrac{1}{2}\zeta_{2\ep}(\tfrac{1}{2}(X^1_t+X^2_t))\right] \\
            &\quad\quad -\bb{E} F(Y) \int_{0}^{1} dx \ f(x) \bb{E}^{\mrm{poly},\theta}_{t,x,x}\left[\tfrac{1}{2}\zeta_{2\ep}(\tfrac{1}{2}(X_t-\tilde{X}_t)) + \tfrac{1}{2}\zeta_{2\ep}(\tfrac{1}{2}(X_t+\tilde{X}_t))\right] \\
            &\quad\quad - \bb{E} F(Y) \int_{0}^{1} dx \ f(x) \bb{E}^{\mrm{poly},\theta}_{t,x} X_t + \bb{E} F(Y) \int_{0}^{1} dx \ f(x) \tfrac{1}{2}\zeta_{2\ep}(x).
        \end{split}
    \end{equation}

    \subsection{Contribution from the boundary data}\label{sec:proof of auxiliary result, subsec:boundary contribution}
    Using the fact that $\alpha + \beta = 0$, one easily computes
    \begin{align*}
        \tfrac{\hat{\alpha}}{2}\partial_x p^{\mrm{neu}}(\ep,X_s;0) + \tfrac{\hat{\beta}}{2}\partial_x p^{\mrm{neu}}(\ep,X_s;1) &= \alpha\left(R_{\ep}'(X_s) - R_{\ep}'(X_s+1)\right) - \tfrac{1}{2}\left(R_{\ep}'(X_s) + R_{\ep}'(X_s+1)\right)
    \end{align*}
    and hence
    \begin{align*}
        \int_{0}^{t} V_{\ep}'(X_s) ds = -\alpha\int_{0}^{t} \left(r_{\ep}''(X_s) - r_{\ep}''(X_s+1)\right) ds + \tfrac{1}{2}\int_{0}^{t} \left(r_{\ep}''(X_s) + r_{\ep}''(X_s+1)\right) ds,
    \end{align*}
    where the functions $R$ and $r$ are defined as in the previous subsection. Since $r_{\ep}(x) - r_{\ep}(x+1)$ and $r_{\ep}(x) + r_{\ep}(x+1)$ are easily checked to define smooth approximations of $\frac{1}{2}\sigma(x)$ and $\frac{1}{2}\zeta(x)$, respectively, we denote them by $\tfrac{1}{2}\sigma_{\ep}(x)$ and $\tfrac{1}{2}\zeta_{\ep}(x)$ and apply It\^{o}'s formula to deduce that
    \begin{align*}
        \int_{0}^{t} V_{\ep}'(X_s) ds &= -\alpha\left(\sigma_{\ep}(X_t) - \sigma_{\ep}(x) - 2\int_{0}^{t} \left(R_{\ep}(X_s) - R_{\ep}(X_s+1)\right) dX_s\right) \\
        &\quad\quad + \left(\tfrac{1}{2}\zeta_{\ep}(X_t) - \tfrac{1}{2}\zeta_{\ep}(x) - \int_{0}^{t} \left(R_{\ep}(X_s) + R_{\ep}(X_s+1)\right) dX_s + X_t\right).
    \end{align*}
    Since the stochastic integration terms are immaterial in the limit $\ep \to 0$ by the remarks following Proposition \ref{prop:key symmetry}, the contribution \eqref{eqn:calculation, boundary part} from the boundary data is then $\alpha \bb{E} F(Y) \int_{0}^{1} f(x) \sigma_{\ep}(x) dx + \Gamma^{\mrm{bdy}} + \mc{E}^{\mrm{bdy}}$, where $\mc{E}^{\mrm{bdy}} \to 0$ as $\ep \to 0$ and 
    \begin{equation}\label{eqn:contribution from boundary data}
        \begin{split}
            \Gamma^{\mrm{bdy}} &= -\alpha\bb{E} F(Y) \int_{0}^{1} dx \ f(x) \bb{E}^{\mrm{poly},\theta}_{t,x} \sigma_{\ep}(X_t) + \bb{E} F(Y) \int_{0}^{1} dx \ f(x) \bb{E}^{\mrm{poly},\theta}_{t,x} \tfrac{1}{2}\zeta_{\ep}(X_t) \\
            &\quad\quad + \bb{E} F(Y) \int_{0}^{1} dx \ f(x) \bb{E}^{\mrm{poly},\theta}_{t,x} X_t - \bb{E} F(Y) \int_{0}^{1} dx \ f(x) \tfrac{1}{2}\zeta_{\ep}(x).
        \end{split}
    \end{equation}

    \subsection{Contribution from the initial data}\label{sec:proof of auxiliary result, subsec:initial data contribution}
    Again, we first compute \eqref{eqn:calculation, initial data part} with $\bb{E}^{\mrm{init}}$ in place of $\bb{E}$ for a fixed realization of the forcing noise $\xi$; at the end, we will take an expectation $\mbf{E}$ with respect to the noise to recover $\bb{E}$. Since $u_0^{\kappa}$ is the derivative of a $\varrho$-invariant function, and $\varrho$ is differentiable with $\varrho' = \sigma$ everywhere except the integers (where the diffusion $X_t$ is not located with probability one), we can write
    \begin{align*}
        \eqref{eqn:calculation, initial data part} &= \bb{E}^{\mrm{init}} F(Y) \int_{0}^{1} dx \ f(x) \bb{E}^{\mrm{poly},\theta}_{t,x} \sigma(X_t) u_0^{\kappa}(\varrho(X_t)) = \bb{E}^{\mrm{init}} F(Y) \int_{0}^{1} dx \ f(x) \bb{E}^{\mrm{poly},\theta}_{t,x} \sigma(X_t) \left(\tilde{\xi}_{\kappa}(\varrho(X_t)) + \alpha\right),
    \end{align*}
    where (by elementary properties of solutions to the heat equation) $\tilde{\xi}_{\kappa} = e^{\kappa\Delta_{\mrm{dir}}} \tilde{\xi}$ for a spatial white noise $\tilde{\xi}$ that is independent of the forcing noise $\xi$, and where $e^{t\Delta_{\mrm{dir}}}$ denotes the heat semigroup generated by the Laplacian on $[0,1]$ with homogeneous Dirichlet boundary conditions. The main term of interest here is of course the term obtained by distributing onto the $\tilde{\xi}_{\kappa}(\varrho(X_t))$, which by expanding the polymer expectation and applying the integration by parts identity Proposition \eqref{prop:malliavin integration by parts} evaluates to
    \begin{align*}
        \bb{E}^{\mrm{init}} \int_{0}^{1} dx \ f(x) \bb{E}^{\mrm{BM}}_{x} \sigma(B_t) \int_{0}^{1} dy \ p^{\mrm{dir}}(\kappa,y; \varrho(B_t)) \tilde{\ms{D}}_y\left(F(Y) \tfrac{e^{-H^{\theta}(t,B)}}{Z^{\theta}(t,x)}\right),
    \end{align*}
    where $p^{\mrm{dir}}(t,x;y) = (e^{t\Delta_{\mrm{dir}}} \delta_y)(x)$ denotes the Dirichlet heat kernel. Exactly as in the calculation of the noise contribution \eqref{eqn:calculation, noise part}, one computes
    \begin{align*}
        \tilde{\ms{D}}_y\left(F(Y) \tfrac{e^{-H^{\theta}}(t,B)}{Z^{\theta}(t,x)}\right) &= -\tfrac{e^{-H^{\theta}}(t,B)}{Z^{\theta}(t,x)} F'(Y) \int_{0}^{1} dx_2 \ f'(x_2) \bb{E}^{\mrm{poly},\theta}_{t,x_2} \mbf{1}_{[0,\varrho(X^2_t)]}(y) \\
        &\quad\quad + \tfrac{e^{-H^{\theta}}(t,B)}{Z^{\theta}(t,x)} F(Y)\left(\mbf{1}_{[0,\varrho(B_t)]}(y) - \bb{E}^{\mrm{poly},\theta}_{t,x} \mbf{1}_{[0,\varrho(\tilde{X}_t)]}(y)\right),
    \end{align*}
    and so by substituting back and taking an expectation $\mbf{E}$ over the noise, we find that the contribution \eqref{eqn:calculation, initial data part} from the initial data is $\Gamma^{\mrm{init}}$, where
    \begin{equation}\label{eqn:contribution from initial data}
        \begin{split}
            \Gamma^{\mrm{init}} &= -\bb{E} F'(Y) \int_{[0,1]^2} dx_1 dx_2 \ f(x_1)f'(x_2) \bb{E}^{\mrm{poly},\theta}_{t,x_1,x_2} \sigma(X^1_t) \int_{0}^{1} dy \ p^{\mrm{dir}}(\kappa,y; \varrho(X^1_t)) \mbf{1}_{[0,\varrho(X^2_t)]}(y) \\
            &\quad\quad + \bb{E} F(Y) \int_{0}^{1} dx \ f(x) \bb{E}^{\mrm{poly},\theta}_{t,x} \sigma(X_t) \int_{0}^{1} dy \ p^{\mrm{dir}}(\kappa,y; \varrho(X_t)) \mbf{1}_{[0,\varrho(X_t)]}(y) \\
            &\quad\quad - \bb{E} F(Y) \int_{0}^{1} dx \ f(x) \bb{E}^{\mrm{poly},\theta}_{t,x,x} \sigma(X_t) \int_{0}^{1} dy \ p^{\mrm{dir}}(\kappa,y; \varrho(X_t)) \mbf{1}_{[0,\varrho(\tilde{X}_t)]}(y) \\
            &\quad\quad + \alpha \bb{E} F(Y) \int_{0}^{1} dx \ f(x) \bb{E}^{\mrm{poly},\theta}_{t,x} \sigma(X_t).
        \end{split}
    \end{equation}
    Substituting our calculations of \eqref{eqn:calculation, noise part}-\eqref{eqn:calculation, initial data part} into \eqref{eqn:lhs of steins}, we conclude finally that
    \begin{align*}
        \bb{E} F(Y)Y = \bb{E} F'(Y) \int_{[0,1]^2} &dx_1 dx_2 \ f(x_1) f'(x_2)\left(\tfrac{1}{2}\zeta_{2\ep}(\tfrac{1}{2}(x_1-x_2)) + \tfrac{1}{2}\zeta_{2\ep}(\tfrac{1}{2}(x_1+x_2))\right) \\
        &+ \Gamma^{\mrm{noise}} + \Gamma^{\mrm{bdy}} + \Gamma^{\mrm{noise}} + \alpha \bb{E}F(Y) \int_{0}^{1} f(x) (\sigma_{\ep}(x)-1) dx + \mc{E}^{\mrm{noise}} + \mc{E}^{\mrm{bdy}}.
    \end{align*}
    The final three terms all vanish as $\ep \to 0$, so we can collectively denote them by $\mc{E}$, giving Proposition \ref{prop:auxiliary result}.

    \section{Proof of Proposition \ref{prop:key symmetry}}\label{sec:key symmetry}

    In this section we prove the key symmetry Proposition \ref{prop:key symmetry}. As a simplifying reduction, we first note that we can replace the upper limit of integration $t$ with $\tau = \inf\{s \geq 0 : X^1_s - X^2_s \in 2{\Z}\} \wedge t$. This is because if $\tau < t$, then the integral can be broken up as the integral from $0$ to $\tau$ plus the integral from $\tau$ to $t$. The latter vanishes because the diffusions $X^1$ and $X^2$ are exchangeable under the polymer expectation after time $\tau$ (by the strong Markov property and 2-periodicity of $R_{\ep}$) while the integral is anti-symmetric. Our goal is therefore to estimate
    \begin{equation}\label{eqn:polymer expectation of stochastic integral}
        \mc{E}_{t,x_1,x_2} = \frac{\bb{E}^{\mrm{BM}}_{x_1,x_2} e^{\sum_{j=1}^{2} \int_{0}^{t} \xi_{\ep}(t-s, X^j_s) - \frac{1}{2}C^j_{\ep}(X^j_s) ds + \int_{0}^{t} V_{\ep}(X^j_s) ds + h_0^{\kappa}(X^j_t)}\int_{0}^{\tau} R_{\ep}(X^1_s - X^2_s) d(X^1_s - X^2_s)}{Z^{\theta}(t,x_1) Z^{\theta}(t,x_2)}
    \end{equation}
    for a fixed realization of the initial data. To keep notation simple, we shall denote by $\mc{X}_{t,x_1,x_2}$ the numerator of \eqref{eqn:polymer expectation of stochastic integral}. We shall also occasionally use the notation $\lesssim$ to mean ``less than or equal to, up to a constant'', with a subscript (if included) indicating what the constant depends on. Our first result is a simple application of Cauchy-Schwarz that leverages uniform estimates on the negative moments of solutions to SHE-type problems to separate the numerator and denominator:

    \begin{lemma}\label{lem:separating numerator and denominator}
        There is a finite constant $C = C(t,\kappa) > 0$ such that $\mbf{E}|\mc{E}_{t,x_1,x_2}| \leq C\left(\mbf{E} \mc{X}_{t,x_1,x_2}^2\right)^{1/2}$.
    \end{lemma}
    \begin{proof}
        By applying Cauchy-Schwarz twice, we have that
        \begin{align*}
            \mbf{E}|\mc{E}_{t,x_1,x_2}| \leq \left(\mbf{E} Z^{\theta}(t,x_1)^{-4}\right)^{1/4} \left(\mbf{E} Z^{\theta}(t,x_2)^{-4}\right)^{1/4} \left(\mbf{E} \mc{X}_{t,x_1,x_2}^2\right)^{1/2}.
        \end{align*}
        Accordingly, it suffices to show that for $x \in [0,1]$ and $p \geq 2$, it holds that $\mbf{E} Z^{\theta}(t,x)^{-p} \leq C$ for some finite $C = C(p,t,\kappa) > 0$. To do so, we compare the negative moments of $Z^{\theta}$ to those of the solution $\mc{Z}^{\theta}$ to \altref{eqn:SHE}{SHE$_{0,0}$} with $\xi_{\ep}$ in place of $\xi$ and $\mc{Z}^{\theta}_0 = e^{h_0^{\kappa}}$, which are easier to bound. Specifically, if $\mc{H}^{\theta} = -\int_{0}^{t} (\xi_{\ep}(t-s,B_s) - \frac{1}{2}C_{\ep}(B_s)) ds - h_0^{\kappa}(B_t)$ is let to denote the analogue of $H^{\theta}$ for this object, then (by Feynman-Kac) $\frac{1}{\mc{Z}^{\theta}(t,x)} e^{-\mc{H}^{\theta}}$ is the density of a Gibbs measure $\bb{P}^{\mrm{Gibbs}}$ and we can write
        \begin{align*}
            Z^{\theta}(t,x) = \mc{Z}^{\theta}(t,x) \bb{E}^{\mrm{Gibbs}} e^{\int_{0}^{t} V_{\ep}(X_s) ds}.
        \end{align*}
        By Jensen's inequality, we then have that
        \begin{align*}
            \mbf{E} Z^{\theta}(t,x)^{-p} = \mbf{E}\left[\mc{Z}^{\theta}(t,x)^{-p} \left(\bb{E}^{\mrm{Gibbs}} e^{\int_{0}^{t} V_{\ep}(X_s) ds}\right)^{-p}\right] \leq \mbf{E}\left[\mc{Z}^{\theta}(t,x)^{-(p+1)} \bb{E}^{\mrm{BM}}_x e^{-\mc{H}^{\theta} - p\int_{0}^{t} V_{\ep}(B_s) ds}\right].
        \end{align*}
        Applying Schwarz again, we obtain the bound
        \begin{equation}\label{eqn:bound for negative moments of Z}
            \mbf{E} Z^{\theta}(t,x)^{-p} \leq \left(\mbf{E} \mc{Z}^{\theta}(t,x)^{-2(p+1)}\right)^{1/2} \left(\mbf{E}(\bb{E}^{\mrm{BM}}_x e^{-\mc{H}^{\theta} - p\int_{0}^{t} V_{\ep}(B_s) ds})^2\right)^{1/2}.
        \end{equation}
        For the second factor, we notice that
        \begin{align*}
            \mbf{E}(\bb{E}^{\mrm{BM}}_x e^{-\mc{H} - p\int_{0}^{t} V_{\ep}(B_s) ds})^2 = \mbf{E}\left[\bb{E}^{\mrm{BM}}_{x,x} e^{-\mc{H}_1^{\theta} - \mc{H}_2^{\theta} - p\int_{0}^{t} V_{\ep}(B^1_s) ds - p\int_{0}^{t} V_{\ep}(B^2_s) ds}\right],
        \end{align*}
        where $\mc{H}_1^{\theta}$ and $\mc{H}_2^{\theta}$ denote the versions of $\mc{H}^{\theta}$ corresponding to independent Brownian motions $B^1$ and $B^2$, respectively. By Fubini and the moment generating function for Gaussians, the right-hand side equals
        \begin{align*}
            \bb{E}^{\mrm{BM}}_{x,x} e^{\int_{0}^{t} p^{\mrm{neu}}(2\ep, B^1_s; B^2_s) ds - p\int_{0}^{t} V_{\ep}(B^1_s) ds - p\int_{0}^{t} V_{\ep}(B^2_s) ds + h_0^{\kappa}(B^1_t) + h_0^{\kappa}(B^2_t)},
        \end{align*}
        and the proof of Lemma \ref{lem:bounds for Z and U} (together with further applications of Schwarz) shows that this is bounded by a constant $C = C(t,\kappa) > 0$. We have thus reduced the problem to bounding the first factor in \eqref{eqn:bound for negative moments of Z}, i.e. the negative moments of $\mc{Z}^{\theta}$. For this, we appeal to the argument of \cite[Theorem 2]{mueller2007regularitydensitystochasticheat}. Namely, it is sufficient to show that we have the following large deviation and comparison lemmas:
        \begin{lemma}
            Let $w(t,x)$ be a bounded progressively measurable process. Given $\delta > 0$, there exist constants $C_0, C_1 > 0$ (independent of $\ep$) such that for all $\lambda > 0$ and $T > 0$,
            \begin{equation}\label{eqn:large deviations estimate}
                \mbf{P}\left(\sup_{0 \leq t \leq T} \sup_{0 \leq x \leq 1} \left|\int_{0}^{t} ds \int_{0}^{1} dy \ p^{\mrm{neu}}(t-s,x;y) w(s,y) \xi_{\ep}(s,y)\right| > \lambda\right) \leq C_0 \exp\left(-\frac{C_1 \lambda^2}{T^{1/2-\delta}}\right).
            \end{equation}
        \end{lemma}
        \begin{lemma}
            Let $Z$ and $\tilde{Z}$ denote two solutions to \altref{eqn:SHE}{SHE$_{0,0}$} with forcing noise $\xi_{\ep}$ and initial data $Z(0,x) = z_0(x)$ and $\tilde{Z}(0,x) = \tilde{z}_0(x)$, respectively. Suppose that, $\bb{P}^{\mrm{init}}$-a.s., we have that $z_0(x) \leq \tilde{z}_0(x)$ for all $x \in [0,1]$. Then, $\mbf{P}$-a.s., it holds that $Z(t,x) \leq \tilde{Z}(t,x)$ for all $t \geq 0$ and $x \in [0,1]$.
        \end{lemma}
        Since solutions to \altref{eqn:SHE}{SHE$_{0,0}$} with the regularized forcing noise $\xi_{\ep}$ are strictly positive by Feynman-Kac, the comparison lemma is immediate from linearity. For the large deviations lemma, we note that \eqref{eqn:large deviations estimate} is equivalent to the estimate of \cite[Lemma 3]{mueller2007regularitydensitystochasticheat} with $w \equiv 1$ and kernel
        \begin{align*}
            K_{\ep}^{t,x}(s,y) = \int_{0}^{1} dz \ p^{\mrm{neu}}(t-s, x; z) p^{\mrm{neu}}(\ep, z; y) w(s,z).
        \end{align*}
        This kernel satisfies the regularity condition of \cite[Proposition A.1]{Sowers1992}, since
        \begin{align*}
            \norm{K_{\ep}^{t,x} - K_{\ep}^{s,y}}_{L^2(\R_+ \times [0,1])} &= \norm{e^{\ep\Delta_{\mrm{neu}}} p^{\mrm{neu}}(t-\cdot,\cdot ; x) w - e^{\ep\Delta_{\mrm{neu}}} p^{\mrm{neu}}(s-\cdot,\cdot ; y) w}_{L^2(\R_+ \times [0,1])} \\
            &\lesssim_{w} \norm{p^{\mrm{neu}}(t-\cdot,\cdot ; x) - p^{\mrm{neu}}(s-\cdot,\cdot ; y)}_{L^2(\R_+ \times [0,1])}
        \end{align*}
        by contractivity of the heat semigroup and the boundedness hypothesis on $w$. Accordingly, the large deviations lemma follows from \cite[Proposition A.2]{Sowers1992}.
    \end{proof}

    Our next result is an immediate consequence of the moment generating function for a sum of Gaussians upon realizing that the noise expectation $\mbf{E}$ only ``hits'' the stochastic integral terms $\int_{0}^{t} \xi_{\ep}(t-s,X^j_s) ds$, which for fixed realizations of the underlying Brownian motions are centered Gaussians with covariance $\Sigma_{ij} = \int_{0}^{t} p^{\mrm{neu}}(2\ep, X^i_s;X^j_s)$:

    \begin{lemma}\label{lem:expectation of numerator squared}
        We have that
        \begin{align*}
            \mbf{E} \mc{X}^2_{t,x_1,x_2} = \bb{E}^{\mrm{BM}}_{x_1,x_2,x_1,x_2} e^{\int_{0}^{t} \bm{R}_{\ep} (\bm{X}_s) ds + \int_{0}^{t} \bm{V}_{\ep}(\bm{X}_s) ds + \bm{h}_0^{\kappa}(\bm{X}_t)} \prod_{j=1}^{2} \int_{0}^{\tau_j} R_{\ep}(X^{2j-1}_s - X^{2j}_s) d(X^{2j-1}_s - X^{2j}_s),
        \end{align*}
        where $\bm{X} = (X^j)_{j=1}^{4}$, the $\bm{R}_{\ep}$, $\bm{V}_{\ep}$, and $\bm{h}_0^{\kappa}$ are 2-periodic functions of a variable $\bm{x} \in \R^4$ defined by
        \begin{align*}
            \bm{R}_{\ep}(\bm{x}) = \sum_{1 \leq i < j \leq 4} p^{\mrm{neu}}(2\ep, x_i; x_j), \quad \bm{V}_{\ep}(\bm{x}) = \sum_{j=1}^{4} V_{\ep}(x_j), \quad \bm{h}_0^{\kappa}(\bm{x}) = \sum_{j=1}^{4} h_0^{\kappa}(x_j),
        \end{align*}
        and the $\tau_j$ are defined analogously to $\tau$ as $\tau_j = \inf\{s \geq 0 : X^{2j-1}_s - X^{2j}_s = 0\} \wedge t$.
    \end{lemma}

    As a next step in the analysis, we view the factor $e^{\int_{0}^{t} \bm{R}_{\ep} (\bm{X}_s) ds + \int_{0}^{t} \bm{V}_{\ep}(\bm{X}_s) ds + \bm{h}_0^{\kappa}(\bm{X}_t)}$ from Lemma \ref{lem:expectation of numerator squared} as the weight of a Gibbs measure $\bb{P}^{\mrm{Gibbs}}_{\bm{x}}$ on paths $\bm{X} \in C_{\bm{x}}([0,t]; \R^4)$ started from $\bm{x} \in \R^4$, that is, we define
    \begin{align*}
        d\bb{P}^{\mrm{Gibbs}}_{\bm{x}}(\bm{X}) = \frac{e^{\int_{0}^{t} \bm{R}_{\ep} (\bm{X}_s) ds + \int_{0}^{t} \bm{V}_{\ep}(\bm{X}_s) ds + \bm{h}_0^{\kappa}(\bm{X}_t)}}{\bb{E}^{\mrm{BM}}_{\bm{x}} e^{\int_{0}^{t} \bm{R}_{\ep} (\bm{X}_s) ds + \int_{0}^{t} \bm{V}_{\ep}(\bm{X}_s) ds + \bm{h}_0^{\kappa}(\bm{X}_t)}} d\bb{P}^{\mrm{BM}}_{\bm{x}}(\bm{X}).
    \end{align*}
    To see why this is useful, notice that by the classical Feynman-Kac formula the denominator above is $\bm{Z}(t,\bm{x})$ for $\bm{Z}$ solving the following heat problem on $\R_+ \times \bb{R}^4$:
    \begin{align*}
        \partial_t \bm{Z} = \frac{1}{2}\Delta \bm{Z} + (\bm{R}_{\ep} + \bm{V}_{\ep})\bm{Z}, \quad \bm{Z}(0,\bm{x}) = e^{\bm{h}_0^{\kappa}(\bm{x})},
    \end{align*}
    (equivalently, since $\bm{R}_{\ep}$, $\bm{V}_{\ep}$, and $\bm{h}_0^{\kappa}$ are all $\varrho$-invariant in each of their spatial variables, we can replace $\R^4$ with $[0,1]^4$ and impose homogeneous Neumann boundary conditions). In particular, $\bm{h}(s,\bm{x}) := \log \bm{Z}(s,\bm{x})$ solves
    \begin{align*}
        \partial_t \bm{h} = \frac{1}{2}\Delta \bm{h} + \frac{1}{2}|\nabla \bm{h}|^2 + (\bm{R}_{\ep} + \bm{V}_{\ep}), \quad \bm{h}(0,\bm{x}) = \bm{h}_0^{\kappa}(\bm{x}),
    \end{align*}
    so by applying It\^{o}'s formula to $\bm{h}(t-s, \bm{X}_s)$ with $\bm{X}$ given by a Brownian motion on $\bb{R}^4$ started at $\bm{x}$, we deduce that
    \begin{align*}
        \bm{h}_0^{\kappa}(\bm{X}_t) = \bm{h}(t,\bm{x}) - \int_{0}^{t} (\bm{R}_{\ep}(\bm{X}_s) + \bm{V}_{\ep}(\bm{X}_s)) ds + \int_{0}^{t} \bm{U}(t-s,\bm{X}_s) \cdot d\bm{X}_s - \frac{1}{2}\int_{0}^{t} |\bm{U}(t-s,\bm{X}_s)|^2 ds,
    \end{align*}
    where $\bm{U}(s,\bm{x}) = \nabla \bm{h}(s,\bm{x})$. It follows that the Radon-Nikodym derivative can be re-written as
    \begin{align*}
        \frac{e^{\int_{0}^{t} \bm{R}_{\ep} (\bm{X}_s) ds + \int_{0}^{t} \bm{V}_{\ep}(\bm{X}_s) + \bm{h}_0^{\kappa}(\bm{X}_t)}}{\bm{Z}(t,\bm{x})} = e^{\int_{0}^{t} \bm{U}(t-s,\bm{X}_s) \cdot d\bm{X}_s - \frac{1}{2}\int_{0}^{t} |\bm{U}(t-s,\bm{X}_s)|^2 ds},
    \end{align*}
    at which point the Girsanov theorem implies that under $\bb{P}^{\mrm{Gibbs}}_{\bm{x}}$ the path $\bm{X}$ satisfies the diffusion equation
    \begin{align*}
        d\bm{X}_s = \bm{U}_{t-s}(\bm{X}_s) ds + d\bm{B}_s, \quad \bm{X}_0 = \bm{x}.
    \end{align*}
    By letting $\bm{U}_s(\bm{x}) = (U^1_s(\bm{x}), \ldots, U^4_s(\bm{x}))$ and $\bm{B}_s = (B^1_s,\ldots,B^4_s)$, it follows that we can then write
    \begin{equation}\label{eqn:numerator squared gibbs measure representation}
        \begin{split}
            &\frac{\mbf{E} \mc{X}_{t,x_1,x_2}^2}{\bm{Z}(t,x_1,x_2,x_1,x_2)} \\
            &\quad\quad\quad= \bb{E}^{\mrm{Gibbs}}_{x_1,x_2,x_1,x_2} \prod_{j=1}^{2} \int_{0}^{\tau_j} R_{\ep}(X^{2j-1}_s - X^{2j}_s) \left([U_{t-s}^{2j-1}(\bm{X}_s) - U_{t-s}^{2j}(\bm{X}_s)] ds + dB^{2j-1}_s - dB^{2j}_s\right).
        \end{split}
    \end{equation}
    In particular, if we introduce the notation
    \begin{align*}
        A_j = \int_{0}^{\tau_j} R_{\ep}(X^{2j-1}_s - X^{2j}_s)[U^{2j-1}_{t-s}(\bm{X}_s) - U^{2j}_{t-s}(\bm{X}_s)]ds, \quad M_j = \int_{0}^{\tau_j} R_{\ep}(X^{2j-1}_s - X^{2j}_s) d(B^{2j-1}_s - dB^{2j}_s),
    \end{align*}
    then \eqref{eqn:numerator squared gibbs measure representation} becomes
    \begin{align*}
        \frac{\mbf{E} \mc{X}_{t,x_1,x_2}^2}{\bm{Z}(t,x_1,x_2,x_1,x_2)} = \bb{E}^{\mrm{Gibbs}}_{x_1,x_2,x_1,x_2} A_1A_2 + \bb{E}^{\mrm{Gibbs}}_{x_1,x_2,x_1,x_2} A_1M_2 + \bb{E}^{\mrm{Gibbs}}_{x_1,x_2,x_1,x_2} A_2M_1 + \bb{E}^{\mrm{Gibbs}}_{x_1,x_2,x_1,x_2} M_1M_2,
    \end{align*}
    which by Cauchy-Schwarz, symmetry, and the fact that $M_1M_2$ is a martingale under the Gibbs measure, is bounded by
    \begin{equation}\label{eqn:bound for gibbs measure representation}
        \bb{E}^{\mrm{Gibbs}}_{x_1,x_2,x_1,x_2} A_1^2 + 2\left(\bb{E}^{\mrm{Gibbs}}_{x_1,x_2,x_1,x_2} A_1^2\right)^{1/2}\left(\bb{E}^{\mrm{Gibbs}}_{x_1,x_2,x_1,x_2} M_2^2\right)^{1/2}.
    \end{equation}
    To now estimate this, we invoke the following lemma.

    \begin{lemma}\label{lem:bounds for Z and U}
        Given $T > 0$, there is a finite positive constant $C = C(T,\kappa)$ such that, for all $t \in (0,T]$ and $\bm{x} \in \R^4$, we have that
        \begin{align*}
            C^{-1} \leq \bm{Z}(t,\bm{x}) \leq C \quad \text{and} \quad |\bm{U}(t,\bm{x})| \leq C |\log\ep|
        \end{align*}
        with the latter holding for $\ep$ sufficiently small.
    \end{lemma}
    \begin{proof}
        It is sufficient to find individual positive constants $C_1 = C_1(T,\kappa)$, $C_2 = C_2(T,\kappa)$, and $C_3 = C_3(T,\kappa)$ such that $C_1^{-1} \leq \bm{Z}(t,\bm{x}) \leq C_2$ and $|\bm{U}_t(\bm{x})| \leq C_3|\log\ep|$ (for $\ep$ sufficiently small), as we can then take $C = \max(C_1,C_2,C_3)$. With this in mind, we begin by bounding $\bm{Z}(t,\bm{x})$. By iteratively applying Cauchy-Schwarz to the Feynman-Kac representation of $\bm{Z}(t,\bm{x})$, we see that it suffices to bound the quantities
        \begin{align*}
            \bb{E}^{\mrm{BM}}_{x_1,x_2} e^{\lambda\int_{0}^{t} p^{\mrm{neu}}(2\ep,X^1_s;X^2_s) ds}, \quad \bb{E}^{\mrm{BM}}_x e^{\lambda\int_{0}^{t} p^{\mrm{neu}}(\ep,X_s;0) ds}, \quad \bb{E}^{\mrm{BM}}_x e^{\lambda\int_{0}^{t} p^{\mrm{neu}}(\ep,X_s;1) ds}
        \end{align*}
        by a constant $C = C(T) > 0$ for any $\lambda \in \R$. For this, we consider only the first quantity, as an identical argument establishes the same bound for the latter two quantities. Moreover, note that we may as well assume $\lambda > 0$, since otherwise we can bound by 1. With this in mind, we expand the exponential and order the time integrations to get
        \begin{equation}\label{eqn:exp moments of intersection local times}
            \bb{E}^{\mrm{BM}}_{x_1,x_2} e^{\lambda\int_{0}^{t} p^{\mrm{neu}}(2\ep,X^1_s;X^2_s) ds} = \sum_{n \geq 0} \lambda^n \bb{E}^{\mrm{BM}}_{x_1,x_2} \int_{\Delta_n} ds_1 \cdots ds_n \ \prod_{i=1}^{n} p^{\mrm{neu}}(2\ep,X^1_{s_i};X^2_{s_i}),
        \end{equation}
        where $\Delta_n = \{0 < s_n < \cdots < s_1 < t\}$. Next, we note that since $p^{\mrm{neu}}$ is $\varrho$-invariant in each of its spatial variables, we can replace the underlying Brownian motions $X^j$ on $\R$ with reflecting Brownian motions $\tilde{X}^j = \varrho(X^j)$ on $[0,1]$ started at $\tilde{x}_j = \varrho(x_j) \in [0,1]$. By then expanding the expectation and iteratively applying Chapman-Kolmogorov together with the bound $p^{\mrm{neu}}(s,x;y) \leq C s^{-1/2}$ for a constant $C > 0$ that does not depend on $x$ or $y$, one obtains
        \begin{align*}
            \bb{E}^{\mrm{RBM}}_{\tilde{x}_1,\tilde{x}_2} \int_{\Delta_n} ds_1 \cdots ds_n \ \prod_{i=1}^{n} p^{\mrm{neu}}(2\ep,\tilde{X}^1_{s_i}; \tilde{X}^2_{s_i}) \leq C^n \int_{\Delta_n} ds_1 \cdots ds_n s_n^{-1/2} \prod_{j=2}^{n} (s_{j-1} - s_j)^{-1/2} = \frac{C^n (\pi t)^{n/2}}{\Gamma(n/2+1)}.
        \end{align*}
        Substituting this bound into \eqref{eqn:exp moments of intersection local times}, we obtain a convergent series that is increasing termwise in $t$, so that $\bb{E}^{\mrm{BM}}_{x_1,x_2} e^{\lambda\int_{0}^{t} p^{\mrm{neu}}(2\ep,X^1_s;X^2_s) ds} \leq C(T)$, as desired. For a lower bound on $\bm{Z}(t,\bm{x})$, we simply observe that since $\bm{R}_{\ep}(\bm{x}) \geq 0$ for all $\bm{x} \in \R^4$ and $\bm{h}_0^{\kappa}$ is bounded (as a smooth function on $\R^4$ determined by its values on $[0,1]^4$), the Feynman-Kac representation implies that $\bm{Z}(t,\bm{x}) \geq C(\kappa) \bb{E}^{\mrm{BM}}_{\bm{x}} e^{\int_{0}^{t} \bm{V}_{\ep}(\bm{X}_s) ds}$, which by Jensen is $\geq C(\kappa) e^{\bb{E}^{\mrm{BM}}_x \int_{0}^{t} \bm{V}_{\ep}(\bm{X}_s) ds} > C(T,\kappa) > 0$.
        
        Moving on to $\bm{U} = \nabla \log\bm{Z}$, we begin by noting that since $|\bm{U}_t(\bm{x})|^2 = \sum_{i=1}^{4} |U^i_t(\bm{x})|^2$ it suffices (by symmetry) to bound just $|U^1_t(\bm{x})|$ by a constant $C = C(T,\kappa) > 0$ independently of $\bm{x}$. To this end, we note that from the mild from representation of the heat problem satisfied by $\bm{Z}$, we have
        \begin{align*}
            \partial_{x_1} \bm{Z}(t,\bm{x}) = \int_{[0,1]^4} \partial_{x_1}\bm{p}^{\mrm{neu}}(t,\bm{x};\bm{y}) e^{\bm{h}_0^{\kappa}(\bm{y})} d\bm{y} + \int_{0}^{t} ds \int_{[0,1]^4} d\bm{y} \ \partial_{x_1}\bm{p}^{\mrm{neu}}(t-s, \bm{x};\bm{y}) (\bm{R}_{\ep}(\bm{y}) + \bm{V}_{\ep}(\bm{y})) \tilde{\bm{Z}}(s,\bm{y}),
        \end{align*}
        where $\bm{p}^{\mrm{neu}}(t,\bm{x};\bm{y}) = \prod_{j=1}^{4} p^{\mrm{neu}}(t,x_j;y_j)$. By noting that $\int_{0}^{1} \partial_x p^{\mrm{neu}}(t,x_1;y_1) dy_1 = 0$, we can re-write the first term with $e^{\bm{h}_0^{\kappa}(\bm{y})} - e^{\bm{h}_0^{\kappa}(x_1,y_2,y_3,y_4)}$ in place of $e^{\bm{h}_0^{\kappa}(\bm{y})}$ and apply a Lipschitz bound to get, for constants $C,c > 0$,
        \begin{align*}
            \left|\int_{[0,1]^4} \partial_{x_1}\bm{p}^{\mrm{neu}}(t,\bm{x};\bm{y}) e^{\bm{h}_0^{\kappa}(\bm{y})} d\bm{y}\right| \lesssim_{\kappa} \int_{0}^{1} |\partial_{x_1} p^{\mrm{neu}}(t,x_1;y_1)| |x_1-y_1| dy_1 \leq C t^{-1}\int_{\R} |u| e^{-c u^2/t} du = \tilde{C}.
        \end{align*}
        As for the second term, our work for $\bm{Z}$ implies that it is
        \begin{align*}
            \lesssim_{T,\kappa} \int_{0}^{t} ds \int_{[0,1]^4} d\bm{y} \ |\partial_{x_1}\bm{p}^{\mrm{neu}}(t-s, \bm{x}; \bm{y})| (\bm{R}_{\ep}(\bm{y}) + |\bm{V}_{\ep}(\bm{y})|).
        \end{align*}
        We now bound the term obtained here by distributing the integrals onto $\bm{R}_{\ep}(\bm{y})$; an identical argument establishes the same bound for the other term. To this end, we note that by making the change of variables $s \mapsto t-s$, expanding $\bm{p}^{\mrm{neu}}(s,\bm{x}; \bm{y}) = \prod_{j=1}^{4} p^{\mrm{neu}}(s,x_j;y_j)$ and $\bm{R}_{\ep}(\bm{y}) = \sum_{1 \leq i < j \leq 4} p^{\mrm{neu}}(2\ep,y_i;y_j)$, and using the fact that $|\partial_x p^{\mrm{neu}}(s,x;y)| \lesssim s^{-1/2} p^{\mrm{neu}}(2s,x;y)$ (this follows from the fact that $|x| \leq e^{x^2/4}$ for all $x \in \R$, from which it is not hard to see that $|\partial_x p(s,x)| \lesssim s^{-1/2} p(2s,x)$; the extension of this identity to $p^{\mrm{neu}}$ is then immediate from the expansion \eqref{eqn:neumann heat kernel}), this term is bounded (up to a constant) by
        \begin{align*}
            \sum_{1 \leq i < j \leq 4} \int_{0}^{t} ds \ s^{-1/2} \int_{[0,1]^4} d\bm{y} \ p^{\mrm{neu}}(2\ep,y_i;y_j) p^{\mrm{neu}}(2s,x_1;y_1)\prod_{k=2}^{4} p^{\mrm{neu}}(s,x_k;y_k).
        \end{align*}
        At this point, the integrations in $\bm{y}$ can be computed exactly via Chapman-Kolmogorov. For instance, in the case $(i,j) = (1,2)$, we get
        \begin{align*}
            \int_{[0,1]^2} dy_4 dy_3 \ &p^{\mrm{neu}}(s,x_3;y_3)p^{\mrm{neu}}(s, x_4;y_4) \int_{[0,1]^2} dy_2 dy_1 \ p^{\mrm{neu}}(2\ep,y_1;y_2) p^{\mrm{neu}}(2s,x_1;y_1) p^{\mrm{neu}}(s,x_2;y_2) \\
            &= \int_{0}^{1} p^{\mrm{neu}}(s+2\ep,x_2;y_1) p^{\mrm{neu}}(2s,x_1;y_1) dy_1 = p^{\mrm{neu}}(3s+2\ep,x_1;x_2).
        \end{align*}
        In any case, we obtain a bound of $(s+\ep)^{-1/2}$ and so $|\partial_{x_1} \bm{Z}(t,\bm{x})| \leq C(\kappa) + \tilde{C}(T,\kappa)\int_{0}^{t} s^{-1/2} (s+\ep)^{-1/2} ds$, which for sufficiently small $\ep$ is $\lesssim_{T,\kappa} |\log \ep|$. Together with our previous bounds for $\bm{Z}$, we conclude finally that $|U^1_t(\bm{x})| = \bm{Z}(t,\bm{x})^{-1} |\partial_{x_1} \bm{Z}(t,\bm{x})| \lesssim_{T,\kappa} |\log \ep|$, completing the proof.
    \end{proof}

    Using Lemma \ref{lem:bounds for Z and U}, we now bound the terms appearing in \eqref{eqn:bound for gibbs measure representation}. Beginning with $\bb{E}^{\mrm{Gibbs}}_{x_1,x_2,x_1,x_2} A_1^2$, we have
    \begin{align*}
        \bb{E}^{\mrm{Gibbs}}_{x_1,x_2,x_1,x_2} A_1^2 \lesssim_{t,\kappa} |\log\ep|^2 \bb{E}^{\mrm{BM}}_{x_1,x_2,x_1,x_2} e^{\int_{0}^{t} \bm{R}_{\ep} (\bm{X}_s) ds + \int_{0}^{t} \bm{V}_{\ep}(\bm{X}_s) ds + \bm{h}_0^{\kappa}(\bm{X}_t)} \left(\sint_{0}^{\tau_1} R_{\ep}(X^1_s - X^2_s)ds \right)^2.
    \end{align*}
    By applying Cauchy-Schwarz to the second factor and bounding the components separately as we did for $\bm{Z}$ in the proof of Lemma \ref{lem:bounds for Z and U}, we deduce that
    \begin{equation}\label{eqn:bound for A_1}
        \bb{E}^{\mrm{Gibbs}}_{x_1,x_2,x_1,x_2} A_1^2 \lesssim_{t,\kappa} |\log\ep|^2 \left(\bb{E}^{\mrm{BM}}_{x_1,x_2} \left(\sint_{0}^{\tau_1} R_{\ep}(X^1_s - X^2_s)ds\right)^4\right)^{1/2}
    \end{equation}
    provided $\ep$ is sufficiently small. As for $\bb{E}^{\mrm{Gibbs}}_{x_1,x_2,x_1,x_2} M_2^2$, by It\^{o}'s isometry we have 
    \begin{align*}
        \bb{E}^{\mrm{Gibbs}}_{x_1,x_2,x_1,x_2} M_2^2 = \bb{E}^{\mrm{Gibbs}}_{x_1,x_2,x_1,x_2} \int_{0}^{\tau_2} R_{\ep}^2(X^3_s - X^4_s) ds,
    \end{align*}
    at which point we can employ the same line of argumentation (i.e. Lemma \ref{lem:bounds for Z and U} and Cauchy-Schwarz) to obtain
    \begin{equation}\label{eqn:bound for M_2}
        \bb{E}^{\mrm{Gibbs}}_{x_1,x_2,x_1,x_2} M_2^2 \lesssim_{t,\kappa} \left(\bb{E}^{\mrm{BM}}_{x_1,x_2} \left(\sint_{0}^{\tau_2} R_{\ep}^2(X^3_s - X^4_s)ds\right)^2\right)^{1/2}.
    \end{equation}
    Putting together everything we've established so far, we have (by Lemmas \ref{lem:separating numerator and denominator}, \ref{lem:bounds for Z and U}, and \eqref{eqn:numerator squared gibbs measure representation})
    \begin{equation}\label{eqn:stochastic integral estimates everything together}
        \mbf{E}|\mc{E}_{t,x_1,x_2}| \lesssim_{t,\kappa} \left(\tfrac{\mbf{E} \mc{X}_{t,x_1,x_2}^2}{\bm{Z}(t,x_1,x_2,x_1,x_2)}\right)^{1/2} \lesssim_{t,\kappa} \left(\bb{E}^{\mrm{Gibbs}}_{x_1,x_2,x_1,x_2} A_1^2 + 2\left(\bb{E}^{\mrm{Gibbs}}_{x_1,x_2,x_1,x_2} A_1^2\right)^{1/2}\left(\bb{E}^{\mrm{Gibbs}}_{x_1,x_2,x_1,x_2} M_2^2\right)^{1/2}\right)^{1/2}
    \end{equation}
    for all sufficiently small $\ep$, where the terms on the right-hand side are bounded as in \eqref{eqn:bound for A_1} and \eqref{eqn:bound for M_2}. We have thus reduced the problem to estimating
    \begin{align*}
        \bb{E}^{\mrm{BM}}_{x_1,x_2} \left(\sint_{0}^{\tau_1} R_{\ep}(X^1_s - X^2_s)ds\right)^4 \quad \text{and} \quad \bb{E}^{\mrm{BM}}_{x_1,x_2} \left(\sint_{0}^{\tau_1} R_{\ep}^2(X^1_s - X^2_s)ds\right)^2.
    \end{align*}
    We claim that
    \begin{equation}\label{eqn:moment estimates of integral of R}
        \begin{split}
            \bb{E}^{\mrm{BM}}_{x_1,x_2} \left(\sint_{0}^{\tau_1} R_{\ep}(X^1_s - X^2_s)ds\right)^4 \lesssim_t \ep^{2}|\log \ep|^4 \quad \text{and} \quad \bb{E}^{\mrm{BM}}_{x_1,x_2} \left(\sint_{0}^{\tau_1} R_{\ep}^2(X^1_s - X^2_s)ds\right)^2 \lesssim_t |\log \ep|^2.
        \end{split}
    \end{equation}
    Indeed, under $\bb{P}^{\mrm{BM}}_{x_1,x_2}$ the process $B = X^1 - X^2$ is a variance 2 Brownian motion started at $x = x_1 - x_2$. Accordingly, we can re-write the former as
    \begin{align*}
        \bb{E}_x \left(\sint_{0}^{\tau} R_{\ep}(B_s) ds\right)^4,
    \end{align*}
    where now $\tau = \inf\{s \geq 0 : B_s \in 2{\Z}\} \wedge t$ (we have been deliberate to use the notation $\bb{E}_x$ instead of $\bb{E}^{\mrm{BM}}_x$ since $B$ is not of unit variance, but this is rather unimportant). By decomposing $R_{\ep} = R_{\ep} \mbf{1}_{R_{\ep} \geq \ep} + R_{\ep} \mbf{1}_{R_{\ep} < \ep}$ and noting that for $\gamma = \ep^{1/2} |\log\ep|^{1/2}$ there is a constant $C > 0$ independent of $\ep$ so that $R_{\ep} \mbf{1}_{R_{\ep} \geq \ep} \leq C\ep^{-1/2}\sum_{n \in \Z}\mbf{1}_{[2n-C\gamma, 2n+C\gamma]}$, we deduce that
    \begin{equation}\label{eqn:bound for fourth moment of integral of R}
        \bb{E}_x \left(\sint_{0}^{\tau} R_{\ep}(B_s) ds\right)^4 \lesssim t^4 \ep^4 + \ep^{-2} \bb{E}_x \left(\sint_{0}^{\tau} \ssum_{n \in \Z} \mbf{1}_{[2n-C\gamma, 2n+C\gamma]}(B_s) ds\right)^4.
    \end{equation}
    Now, by definition, there are only two intervals of the form $[2n-C\gamma,2n+C\gamma]$ that $B$ can possibly enter up to time $\tau$. These intervals are determined uniquely by the starting point $x$. To bound the second term, it thus suffices (by Minkowski's inequality and translation invariance of Brownian motion) to bound $\bb{E}_x \left(\sint_{0}^{\tau} \mbf{1}_{[-C\gamma, C\gamma]}(B_s) ds\right)^4$. To do so, we note that by symmetry and the strong Markov property, we have
    \begin{align*}
        \bb{E}_x \left(\sint_{0}^{\tau} \mbf{1}_{[-C\gamma, C\gamma]}(B_s) ds\right)^4 \lesssim \bb{E}_0 \left(\sint_{0}^{\tau_{C\gamma}} \mbf{1}_{[0, C\gamma]}(B_s) ds\right)^4
    \end{align*}
    for $\tau_{C\gamma} = \inf\{s \geq 0 : B_s = C\gamma\} \wedge t$. By changing variables $s \mapsto \gamma^2 s$ and applying Brownian scaling, the right-hand side becomes
    \begin{align*}
        \gamma^8\bb{E}^{\mrm{BM}}_0 \left(\sint_{0}^{\tau_C} \mbf{1}_{[0, C]}(\tilde{B}_s) ds\right)^4 = \ep^4 |\log\ep|^4 \bb{E}^{\mrm{BM}}_0 \left(\sint_{0}^{\tau_C} \mbf{1}_{[0, C]}(\tilde{B}_s) ds\right)^4,
    \end{align*}
    where $\tilde{B}$ is another variance 2 Brownian motion and $\tau_C = \inf\{s \geq 0 : \tilde{B}_s = C\} \wedge (\gamma^{-2} t)$. The same argument used in \cite{gu2024integrationpartsinvariantmeasure} implies that $\bb{E}^{\mrm{BM}}_0 (\int_{0}^{\tau_C} \mbf{1}_{[0, C]}(\tilde{B}_s))^4 \lesssim 1$, so that \eqref{eqn:bound for fourth moment of integral of R} becomes
    \begin{align*}
        \bb{E}^{\mrm{BM}}_x \left(\sint_{0}^{\tau} R_{\ep}(B_s) ds\right)^4 \lesssim_t \ep^2 |\log\ep|^4,
    \end{align*}
    which is the first estimate in \eqref{eqn:moment estimates of integral of R}. The second estimate is proved by an identical argument. Substituting these estimates into \eqref{eqn:stochastic integral estimates everything together}, we obtain Proposition \ref{prop:key symmetry}.
    
    \printbibliography
\end{document}